\documentclass[hyperref]{article}
\usepackage{natbib}
\usepackage[top=3cm,bottom=2cm,left=1.6cm,right=1.6cm]{geometry} % 页边距
\usepackage{graphicx}
\usepackage{lineno}
\usepackage{amsmath,amssymb,amsfonts}
\usepackage{arydshln}
\usepackage{caption}
\usepackage{dsfont}	
\usepackage[linesnumbered,boxed,ruled,commentsnumbered]{algorithm2e}
\usepackage{subfigure}
\usepackage{textcomp}
\usepackage{url}
\usepackage{float}
\DeclareGraphicsExtensions{.pdf,.jpeg,.png}
\usepackage{epstopdf}
\usepackage{arydshln}
\usepackage{dsfont}
\usepackage{booktabs}
\usepackage{autobreak}
\usepackage{caption}
\usepackage{bm}
\usepackage{mfirstuc}
\usepackage{enumitem} 
\usepackage[colorlinks,linkcolor=blue,anchorcolor=blue,citecolor=black]{hyperref}
\newtheorem{theorem}{Theorem}
\newtheorem{example}{Example}[section]
\newtheorem{proof}{Proof}
\begin{document}
	\title{\bf Optimal Sliced Latin Hypercube Designs with Slices of Arbitrary Run Sizes} 
	
	\author{\sffamily Jing Zhang$^1$,Jin Xu $^{1}$, Kai Jia $^1$, Yimin Yin $^1$ and Zhengming Wang$^2$\\
		{\sffamily\small $^1$ College of Liberal Arts and Sciences, National University of Defense Technology}\\
		{\sffamily\small $^2$ College of Advanced Interdisciplinary Research, National University of Defense Technology }}
	%\date{\ukdate\today}
	\maketitle

%\begin{document}
%	\title{Optimal Sliced Latin Hypercube Designs with Slices of Arbitrary Run Sizes}
%	\date{\ukdate\today}
%	\author{Jing Zhang\\
%		College of Liberal Arts and Sciences, National University of Defense Technology\\
%		and\\
%		Jin Xu \\
%		College of Liberal Arts and Sciences, National University of Defense Technology\\
%		and\\
%		Kai Jia \\
%		College of Liberal Arts and Sciences, National University of Defense Technology\\
%		and\\
%		Yimin Yin\\
%		College of Liberal Arts and Sciences, National University of Defense Technology\\
%		and\\
%		Zhengming Wang\\
%		College of Advanced Interdisciplinary Research, National University of Defense Technology
%	}
%	\maketitle

\begin{abstract}
		Sliced Latin hypercube designs (SLHDs) are widely used in computer experiments with both quantitative and qualitative factors and in batches. Optimal SLHDs achieve better space-filling property on the whole experimental region. However, most existing methods for constructing optimal SLHDs have restriction on the run sizes. In this paper, we propose a new method for constructing SLHDs with arbitrary run sizes, and a new combined space-filling measurement describing the space-filling property for both the whole design and its slices. Furthermore, we develop general algorithms to search the optimal SLHD with arbitrary run sizes under the proposed measurement. Examples are presented to illustrate that effectiveness of the proposed methods.
\end{abstract}

\section{Introduction}
\label{Section Introduction}
Computer experiments are becoming increasingly significant in many fields, such as finite element analysis and computational fluid dynamics. Latin hypercube designs (LHDs) \cite{McKay1979} are widely used in computer experiments because of their  optimal univariate uniformity. A design with $n$ runs and $q$ factors is called an LHD, if the design is projected onto any one dimension, there is precisely one point lying within one of the $n$ intervals $(0,1/n],(1/n,2/n],\cdots,((n-1)/n,1]$. Such an LHD is said to have optimal univariate uniformity. Sliced Latin hypercube designs (SLHDs) are LHDs that can be partitioned into some LHD slices \cite{Peter2012Sliced}, which means that the SLHDs have the optimal univariate uniformity for both the whole design and their slices.  In \cite{he2016central}, a central limit theorem for SLHDs is proposed.  SLHDs are popular for computer experiments with both qualitative and quantitative variables; see \cite{Qian2008Gaussian,Gang2009Prediction,Deng2017Additive} and the references therein.  Each slice of an SLHD can be used under one  level-combination of the qualitative factors.  However, the original  SLHDs  and almost all existing methods for constructing variants of SLHDs requires that the run sizes of each slice are equal;  see \cite{hwang2016sliced,yin2014sliced,xie2014general,yang2016construction}.

% Thus, the design points under each level-combination of the qualitative factors should spread evenly over the experimental region when the response surfaces at different level combinations of the qualitative factors are similar [11].

An SLHD is called desirable if its design points are well spread out for both the whole design and its slices. Randomly generated SLHDs usually have a poor space-filling property in the entire experimental region, i.e., randomly generated SLHDs may not be desirable.  There are a lot of methods that aim to improve the space-filling property of an SLHD. For instance, the method proposed by \cite{Huang2015Computer} can be used to generate an optimal clustered-sliced Latin hypercube design (OCSLHD) which has good space-filling property in the whole experimental region. 
%The design points of each slice with a good  space-filling property in  the experimental region
In a multi-fidelity computer experiment,  each slice of an OCSLHD can be used for each accuracy of the computer code \cite{Huang2015Computer}.  Generally, we want to use more design points for the lower-accuracy experiments than those of the higher-accuracy experiments, since the lower the accuracy is, the faster it runs \cite{Kennedy2000Predicting, Qian2006Buildingsurrogate}.  However,  a lot of existing method for constructing optimal SLHDs can only generate SLHDs with equal run sizes of each slices, e.g., \cite{Huang2015Computer,Ba2015Optimal,Chen2016}. 
%To overcome this restriction, we need a method that can construct SLHDs with both slices of arbitrary run sizes and a good space-filling property. Many researchers  try to provide SLHDs with flexible run sizes. For example, the method provided by \cite{kong2018flexible} can generate flexible sliced designs, but the designs are not LHDs. The method given in \cite{xu2018sliced} can generate SLHDs with unequal batch sizes, but the generated designs only accommodate two different run sizes. Recently, \cite{2019arXiv190502721X} proposed an algorithm that can construct SLHDs with unequal run sizes, but their method is difficult to search the optimal design.
To overcome this restriction, we need a method that can construct SLHDs with slices of arbitrary run sizes, and with a good space-filling property over the whole experimental region.  For example, reference \cite{kong2018flexible} gave flexible sliced designs, but such designs are not LHDs. The method given
in \cite{xu2018sliced} provided SLHDs with unequal batch sizes, but this type of design only accommodates two different run sizes.  An algorithm is proposed  in \cite{2019arXiv190502721X} to construct an SLHD with unequal run sizes, but this construction method is difficult  to search the optimal design. 

In this paper, we propose a method to construct SLHDs with slices of arbitrary run sizes, which are called flexible sliced Latin hypercube designs (FSLHDs). The new construction method can be easily adapted to generate the optimal design. Furthermore, we provide a combined space-filling measurement (CSM) to descibe  the space-filling properties of both the whole design and each of slices. Based on an optimization algorithm called the enhanced stochastic evolutionary algorithm (ESE), we propose a sliced ESE (SESE) algorithm to find the optimal FSLHDs. We further develop an efficient two-part algorithm to improve the efficiency in generating space-filling FSLHDs with large runs and factors. The generated optimal FSLHDs have three attractive features:
(i) arbitrary run sizes of all slices, 
(ii) optimal univariate uniformity in the whole design and each slice,
(iii) good space-filling property in the experimental region. We believe that they are suitable for many multi-fidelity computer experiments in practice.

The remainder of this paper is organized as follows. The construction of FSLHDs is provided in Section \ref{Section 2}. In Section \ref{Section 3}, an CSM  are given to descibe  the space-filling properties of both the whole design and each of slices, and then we develop an SESE algorithm to obtain optimal FSLHDs based on the CSM and a two-part algorithm to improve  efficiency. Some simulation results are illustrated in Section \ref{Simulation results}.
Section \ref{Remarks} provides some discussions. Section \ref{Conclusions} concludes this paper. 

\section{Construction of SLHDs with slices of arbitrary run sizes }
\label{Section 2}
For a real number $a$, let $\lceil a\rceil$ denote the smallest integer not smaller than $a$. Given $u$ positive integers $n_{1},\cdots, n_{u}$, let $n=\sum_{i=1}^{u}n_{i}$ and let $L = { \rm lcm}(n_{1},\cdots,n_{u},n)$ be the least common multiple of $n_{1},\cdots,n_{u}$,  and $n$.  Suppose that  ${\rm FSLHD}(n_{1},\cdots,n_{u}; u, q)$ is an FSHLD with $u$ slices of run sizes $n_{1},\cdots,n_{u}$ and $q$ factors. Each column of the FSLHD is generated independently by the following algorithm. 
\begin{enumerate} [itemindent=2em] 
	\item[\textbf{Step 1.}] Let $H^{i}=\varnothing$ for $i=1,\cdots,u$, and $R_{0}=\varnothing$. 
	\item[\textbf{Step 2.}] For $j=1,\cdots,n$, let $R_{j,0}= R_{j-1}\bigcup\{j\}$ and calculate
	\begin{center}
		$\theta_{j}=\sum_{i=1} ^{u}\left(\lceil n_{i}(j+1)/n\rceil- \lceil n_{i}j/n\rceil\right)$.\\
	\end{center}
	If $\theta_{j}>0$, for $k=1,\cdots,\theta_{j}$, let $l$ denote the $k$th smallest integer of the set $ \{p|\lceil n_{p}(j+1)/n\rceil - \lceil n_{p}j/n\rceil=1\}$  and $r = \min \{ r|\lceil n_{l}r/n\rceil = \lceil n_{l}j/n\rceil, r\in R_{j,k-1}\}$
	add $r$ to $H^{l}$ and let $R_{j,k}=R_{j,k-1}\setminus \{r\}$. Let $R_{j}=R_{j,\theta_{j}}$ and go to the next $j$.
	\item[\textbf{Step 3.}]For $i=1,\cdots,u$, generate a vector $\boldsymbol {h}^{i}$ by randomly permuting $H^{i}$.
	\item[\textbf{Step 4.}]For $i=1,\cdots,u$, calculate $\boldsymbol{m}^{i}=L\boldsymbol{h}^{i}/n$, where $\boldsymbol{m}^{i} = (m^{i}_{1},\cdots,m^{i}_{n_{i}})$. Combine $\boldsymbol{m}^{1},\cdots,\boldsymbol{m}^{u}$ to  obtain  an $n$-dimensional column vector $\boldsymbol{m}=(m_{1},\cdots,m_{n})^{T}$, then
	let $\boldsymbol{d}^{i}=(d^{i}_{1},\cdots,d^{i}_{n_{i}})^{T} $ be constructed by
	\begin{equation}
	\label{eq1a}
	d^{i}_{s}=(m^{i}_{s}-\varepsilon^{i}_{s})/L,\quad
	s=1,\cdots,n_{i}, 
	\end{equation}
	where $\varepsilon^{i}_{s}\sim U(0,1)$. Combine $\boldsymbol{d}^{1},\cdots,\boldsymbol{d}^{u}$ to  obtain  an $n$-dimensional column vector $\boldsymbol{d}=(d_{1},\cdots,d_{n})^{T}$,  and $\boldsymbol{d}$ is one column  of the design.
\end{enumerate}

In the above algorithm  $\boldsymbol{m}$ is called a column of  the flexible sliced Latin hypercube (FSLH). The following theorem shows that both the  whole FSLHD and its slices  are LHDs. 
\begin{theorem}
	\label{Theo 1}
	Let $\boldsymbol{d}=(d_{1},\cdots,d_{n})^{T}$ denote an arbitrary column of FSLHD($n_{1},\cdots,n_{u};u, q$) generated by the above method. Let $\boldsymbol{d}^{1},\cdots,\boldsymbol{d}^{u}$ denote each slice.
	For $i =1,\cdots,u$, let $t^{i}=L/n_{i}$ and $t'=L/n$. \\
	(i)Precisely one point of $\boldsymbol{d}=(d_{1},\cdots,d_{n})^{T}$  lies within one of the $n$ intervals  $(0,1/n],(1/n,2/n],\cdots,((n-1)/n,1]$.\\
	(ii)Precisely one point of $\boldsymbol{d}^{i}=(d^{i}_{1},\cdots,d^{i}_{n_{i}})$ lies within one of the $n_{i}$ intervals  $(0,1/n_{i}],(1/n_{i},2/n_{i}]$,$\cdots,((n_{i}-1)/n_{i},1]$.
\end{theorem}

\begin{proof}
	(i)  Combine $\boldsymbol{h}^{1},\cdots,\boldsymbol{h}^{u}$  to obtain  $\boldsymbol{h}=(h_{1},\cdots,h_{n})^{T}$ that is a permutation of $\{1,\cdots,n\} $.
	Combine $\boldsymbol{m}^{1},\cdots,\boldsymbol{m}^{u}$  to obtain  $\boldsymbol{m}=(m_{1},\cdots,m_{n})^{T}$. Therefore, $\boldsymbol{m}= L\boldsymbol{h}/n$.
	For $t' =L/n$, because $\lceil \boldsymbol{m}/t' \rceil=\lceil (L\boldsymbol{h}/n)/t' \rceil =\lceil \boldsymbol{h}\rceil $,  $\lceil \boldsymbol{m}/t' \rceil$ is a  permutation of $\{1,\cdots,n\}$.
	Therefore, precisely one point of $\boldsymbol{d}=(d_{1},\cdots,d_{n})^{T}$  lies within one of the $n$ intervals  $(0,1/n],(1/n,2/n],\cdots,((n-1)/n,1]$.
	
	(ii) According to Step 2, for $i=1,\cdots, u$,  it is  clear that
	$card(H^{i})= \sum_{j=1}^{n} (\lceil n_{i}(j+1)/n\rceil - \lceil n_{i}j/n\rceil)=\lceil n_{i}(n+1)/n\rceil - \lceil n_{i}/n\rceil = n_{i}$, and for $j=1,\cdots,n$, $\lceil n_{i}j/n\rceil < \lceil n_{i}(j+1)/n\rceil $. 
	For any $i$,$j$, there is an integer $h\in H^{i}$ that satisfies $\lceil n_{i}h/n\rceil =\lceil n_{i}j/n\rceil$.
	Therefore, we have $\{m|m=\lceil n_{i}h/n\rceil, h\in H^{i}\}=\{1,\cdots,n_{i}\}$, which means that $\lceil n_{i}\boldsymbol{h}^{i} /n\rceil$ is a permutation of $\{1,\cdots,n_{i}\}$. 
	Since $\boldsymbol{m}^{i}=  L\boldsymbol{h}^{i}/n$, we have $\lceil \boldsymbol{m}^{i}/t^{i}\rceil=\lceil (L\boldsymbol{h}^{i}/n)/(L/n_{i})\rceil=\lceil  n_{i}\boldsymbol{h}^{i}/n\rceil $. Thus, $\lceil \boldsymbol{m}^{i}/t^{i}\rceil$ is a  permutation of $\{1,\cdots,n_{i}\}$. Therefore,  precisely one point of $\boldsymbol{d}^{i}$ lies within one of the $n_{i}$ intervals  $(0,1/n_{i}],(1/n_{i},2/n_{i}],\cdots,((n_{i}-1)/n_{i},1]$.
\end{proof}

We give an example to illustrate the process of the above method.
\begin{example}
	\label{Example 1}
	Consider $n_{1}=3$, $n_{2}=4$, $n_{3}=5$, $u=3$, $n=12$, and $L=60$. 
	\begin{enumerate}  [itemindent=2em] 
		\item[\textbf{Step 1.}]  $H^{1}=H^{2}=H^{3}=R_{0}=\emptyset$.
		\item[\textbf{Step 2.}]  Calculate $(\theta_{1},\cdots,\theta_{n})=(0,1,1,2,0,1,1,1,2,0,0,3)$.
		For $j=1$, then $R_{1,0}=\{1\}$, since $\theta_{1}=0$, we obtain $R_{1}=R_{1,0}=\{1\}$.
		For $j=2$, $R_{2,0}= R_{1}\cup \{2\}=\{1,2\}$, $\theta_{2}=1$, only an integer $l=3$ satisfies
		$\lceil n_{l}(j+1)/n\rceil - \lceil n_{L}j/n\rceil=1$,  and
		$r=$\textrm{min}$ \{r|\lceil n_{3}r/n\rceil =\lceil n_{3}j/n\rceil$, $r\in R_{2,0}\}=$ \textrm{min}$\{1,2\}=1$.
		Hence, we add $r=1$ to  $H^{3}$, $R_{2,1}=R_{2,0}\backslash \{1\}=\{2\}$, and $R_{2}=R_{2,1}=\{2\}$.
		For $j=3$, $R_{3,0}= R_{2}\cup \{3\}=\{2,3\}$, $\theta_{3}=1$,
		only an integer $l=2$  satisfies $\lceil n_{2}(j+1)/n\rceil - \lceil n_{2}j/n\rceil=1$, and $r=$\textrm{min}$ \{r|\lceil n_{2}r/n\rceil =\lceil n_{2}j/n\rceil$, $r\in R_{3,0}\}=$ \textrm{min}$\{2,3\}=2$.
		Therefore, we add $r=2$ to $H^{2}$ , $R_{3,1}= R_{3,0} \backslash \{2\}= \{3\}$, and $R_{3}=R_{3,1}=\{3\}$.\textsc{}
		After passing all $j$, we can get $R_{12}=\emptyset$, $H^{1}=\{3,7,10\}$, $H^{2}=\{2,5,8,11\}$, and $H^{3}=\{1,4,6,9,12\}$.
		\item[\textbf{Step 3.}] 
		We get $\boldsymbol{h}^{1}=(10,7,3)$ , $\boldsymbol{h}^{2}=(5,8,2,11)$, and $\boldsymbol{h}^{3}=(6,9,12,1,4)$ by randomly permuting $H^{1}$ , $H^{2}$, and $H^{3}$. 
		\item[\textbf{Step 4.}] 
		We obtain  $\boldsymbol{m}^{1}=(50,35,15)$, $\boldsymbol{m}^{2}=(25,40,10,55)$, and $\boldsymbol{m}^{3}=(30,45,60,5,20)$.
		Then  $\boldsymbol{d}^{i}=(d^{i}_{1},\cdots,d^{i}_{n_{i}})^{T} $ is constructed  through
		$d^{i}_{s}=(m^{i}_{s}-\varepsilon^{i}_{s})/60$,
		where  $i=1,\cdots,3$, $s=1,\cdots,n_{i}$, and $\varepsilon^{i}_{s}\sim U(0,1)$. Thus, we obtain an arbitrary column $\boldsymbol{d}=(d_{1},\cdots,d_{n})^{T}$  of the design.  
	\end{enumerate}
\end{example}

\section{Optimal  SLHDs with slices of arbitrary run sizes}
\label{Section 3}
Given $n_{1},\cdots,n_{u},u,q$, a number of possible FSLHDs  can be  generated through the proposed
method in Section \ref{Section 2}. Among such FSLHDs, we can find the optimal FSLHD through a given space-filling  criterion. We first propose a combined space-filling  measurement (CSM) to evaluate space-filling  property of FSLHD in Subsection \ref{subsection3.1}. Then, to keep the structure of the design during the optimization process, three methods are proposed to change position of the elements in one column in Subsection \ref{subsection3.2}. Finally, we present a sliced ESE algorithm  to optimize FSLHD  in Subsection \ref{subsection3.3}.  An efficient two-part algorithm for generating the space-filling  FSLHD is given in Subsection \ref{subsection3.4}.

\subsection{A combined space-filling measurement for FSLHD{\scriptsize s} }
\label{subsection3.1}
Various space-filling criteria are used to evaluate the LHDs, such as the maximin distance criterion \cite{Johnson1990Minimax,Grosso2009Finding,Dam2007MaximinLatin, Dam2009Bounds}, the $\phi_{t}$ criterion \cite{Jin2016An, Morris1995Exploratory, Ye2000Algorithmic, Viana2010An}, and the centered $L_{2}$-discrepancy (${\textrm CD}_{2}$) criterion  \cite{Hickernell1998A, Fang2002Centered}. All the  space-filling criteria can be extend to describe  space-filling propert of  the FSLHDs. We mainly focus on the $\phi_{t}$ criterion which is an attractive extension of maximin distance criterion.

The maximin distance criterion is a popular space-filling criterion  introduced in \cite{Johnson1990Minimax}. Let $\boldsymbol{D}=[\boldsymbol{x}_{1},\cdots,\boldsymbol{x}_{n}]^{T}$ denote a design matrix with $n$ runs and $q$ factors, where each row $ \boldsymbol{x}_{i}^{T}=(x_{i1},\cdots,x_{iq})$ is a design point and each column is a factor with $ i=1,\cdots,n$.  A maximin distance design is generated by maximizing the minimum inter-site distance, which is expressed as
\begin{equation}
\label{maximin distance}
\min_{\forall 1\leq i,j\leq n ,i\neq j}d(\boldsymbol{x}_{i},\boldsymbol{x}_{j}),
\end{equation}
where $d(\boldsymbol{x}_{i},\boldsymbol{x}_{j})$ is the distance between the design points $\boldsymbol{x}_{i}$  and $\boldsymbol{x}_{j}$ given by:
\begin{equation}
\label{inter-point distance}
d(\boldsymbol{x}_{i},\boldsymbol{x}_{j})= d_{ij}=\left(\sum_{k=1}^{q}|x_{ik}-x_{jk}|^{m}\right)^{1/m}, \text{ $m=1$ or 2}.
\end{equation}
Here $m=1$ and  $m=2$ are the rectangular and Euclidean distances,  respectively. In this article, we use the Euclidean distance. An extension of the maximin distance criterion \cite{Jin2016An} is given by
\begin{equation}
\label{phi_t value}
\phi_{t}=\left(\sum_{1 \leq i<j\leq n}(d_{ij})^{-t}\right)^{1/t},
\end{equation}
where $t$ is a positive integer. It is obviously that as $t\longrightarrow \infty $,  minimizing  (\ref{phi_t value})  is equivalent to maximizing (\ref{maximin distance}). The calculation of  $\phi_{t}$  is simpler compared with the maximin distance criterion. 

We search an optimal design by minimizing $\phi_{t}$, i.e.
\begin{equation}
\boldsymbol{D}^* = \arg \min_{\boldsymbol{D}}\phi_{t}(\boldsymbol{D}).
\end{equation}
Suppose that $\boldsymbol{D}$ is the design matrix of an ${\rm FSLHD}(n_{1},\ldots,n_{u};u, q)$. For $i=1,\ldots,u$, let $\boldsymbol{D}^{(i)}$ denote each slice of $\boldsymbol{D}$. We need to consider both the space-filling properties of the whole FSLHD and that of its slices. Consequently,	our goal is to find a maximin FSLHD that minimizes $\phi_{t}(\boldsymbol{D})$ for the entire design as well as $\phi_{t}(\boldsymbol{D}^{(i)})$  for each slice of $\boldsymbol{D}$ ($i=1,\ldots,u$). This is a multi-objective optimization problem. It is a common method in multi-objective problem to use a weighted average of all individual objectives.	It motivates us to develop a combined space-filling measurement (CSM) as follows:
\begin{equation}
\label{CSM}
\phi_ {\textrm{CSM}}(\boldsymbol{D})=w\phi_ {t}(\boldsymbol{D})+(1-w)\left(\sum_{i=1}^{u}\lambda_{i}\phi_ {t}(\boldsymbol{D}^{(i)})\right),
\end{equation}
where $\lambda_{i}=n_{i}/n$ , $\sum_{i=1}^{u}\lambda_{i}=1$, and $w \in (0,1)$. 	Since  run sizes of slices  are  $n_{1},\ldots,n_{u}$,  respectively, it makes sense that we take the weight of each slice to be $\lambda_{i}=n_{i}/n$,  for $i=1,\ldots,u$.  The weight $w$ is selected flexibly.  The space-filling property of the whole  FSLHD is more important,  hence we set $w=1/2$ in general. We can define a maximin distance FSLHD with respect to the CSM as the one which minimizes  (\ref{CSM}). 

Note that other space-filing criteria can also evaluate the FSLHD. For instance, we can obtain  an uniform FSLHD by minimizing a similar CSM given by
\begin{equation}
\begin{aligned}
\label{CSM2}
\phi_{\textrm{CSM}}(\boldsymbol{D})=w\phi_{\textrm {CD}_{2}}(\boldsymbol{D})+(1-w)\left(\sum_{i=1}^{u}\lambda_{i}\phi_{\textrm {CD}_{2}}(\boldsymbol{D}^{(i)})\right),
\end{aligned}
\end{equation}
where $\phi_{\textrm{CD}_{2}}$ is  the centered $L_{2}$-discrepancy  defined as
\begin{equation}
\begin{aligned}
\phi_{\textrm{CD}_{2}}
= &\left(\left(\frac{13}{12}\right)^{2}-\frac{2}{n}\sum_{i=1}^{n}\sum_{k=1}^{m}\left(1+\frac{1}{2}|x_{ik}-0.5|-|x_{ik}-0.5|^{2}\right)\right.\\
&\left.+\frac{1}{n^{2}}\sum_{i=1}^{n}\sum_{j=1}^{n}\sum_{k=1}^{m}\left(1+\frac{1}{2}|x_{ik}-0.5|+\frac{1}{2}|x_{jk}-0.5|-|x_{ik}-x_{jk}|\right)\right)^{1/2}
\end{aligned}
\end{equation}
proposed in \cite{Hickernell1998A}. 
%Based on the  CSM,  we can develop a new optimization algorithm to construct the optimal FSLHD. In our proposed program, we select (\ref{CSM}) as a combined space-filling criterion.

\subsection{Exchange procedures for FSLHDs}
\label{subsection3.2}
In the literature, some optimization algorithms have widely used to construct an optimal LHD. They utilize an exchange procedure to iteratively search the optimal LHD  in  the design space. In this way,  two randomly  selected elements in an arbitrary column of an LHD are exchanged to generate a new design. The exchange procedure for an FSLHD is  more complex since the design should keep the sliced structure.
In this subsection, in the optimization process of an FSLHD,  we present three exchange  procedures to generate a neighbor of the design which do not change the sliced structure of the design. A neighbor of an FSLH corresponds to a neighbor of an FSLHD.
Let $\boldsymbol{M}$  be the FSLH($n_{1},\cdots,n_{u};u,q$) constructed in Section \ref{Section 2}. Let $\boldsymbol{M}_{\textrm{N}}$ denote a neighbor of an FSLH and let  $\boldsymbol{D}_{\textrm{N}}$ denote  a neighbor of an FSLHD.

\subsubsection{The within-slice exchange procedure}
\label{subsubsection3.2.1}
Given  an FSLH($n_{1},\cdots,n_{u};u,q$)$(\boldsymbol{M})$, let $n_{0}=0$, $r_{0}=0$, and $r_{i}=\sum_{k=0}^{i}n_{k}$,  for $i=1,\cdots,u$. The within-slice exchange procedure in the $i$th  slice of $\boldsymbol{M}$ is to  draw an  $\boldsymbol{M}_{\textrm{N}}$ by the following four steps:
\begin{enumerate} [itemindent=2em] 
	\item[\textbf{Step 1.}] Randomly select a column of $\boldsymbol{M}$.
	\item[\textbf{Step 2.}] Select any two different elements $d_{j}, d_{k}$ in $i$th  slice of the column, where $r_{i-1}+1\leq j,k\leq r_{i}$.
	\item[\textbf{Step 3.}] Exchange $d_{j}$ and $d_{k}$ in the same slice.
	\item[\textbf{Step 4.}] Generate $\boldsymbol{M}_{\textrm{N}}$.
\end{enumerate}

After this  procedure, the neighbor design  $\boldsymbol{M}_{\textrm{N}}$ still keeps the sliced structure. The within-slice exchange procedure is explained by an example about FSLH(4,6;2,2) illustrated in Figure \ref{Fig1}.
\begin{figure}[h]
	\centering
	\includegraphics[width = 0.4\textwidth]{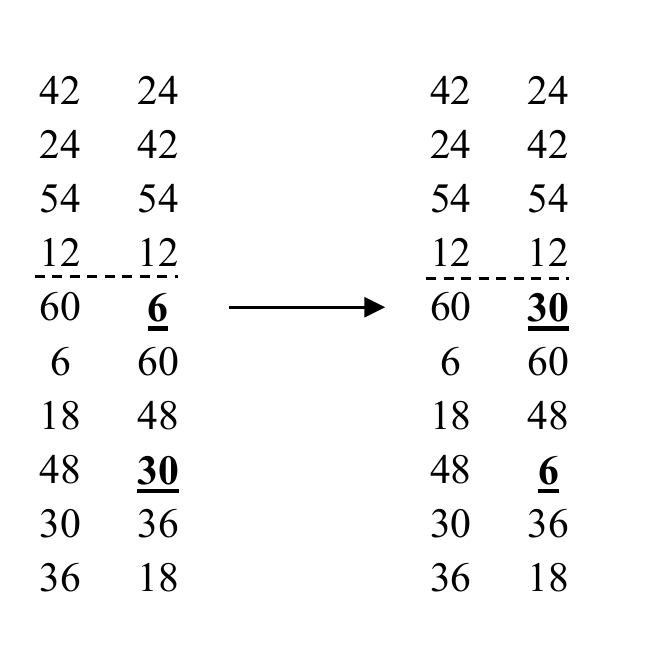}
	\caption{The within-slice exchange procedure. Left: The original FSLH(4,6;2,2). Right: The neighbor of  the  FSLH after  exchanging 6 and 30 in the second slice  and in the second column of the desigh.}
	\label{Fig1}
\end{figure}

\subsubsection{The different-slice exchange and the out-slice exchange procedures}
\label{section3.2.2}
We first give some notations.
Given an FSLH($n_{1},\cdots,n_{u};u,q$)$(\boldsymbol{M})$, let $\boldsymbol{M}(l:m,j)$ denote the $l$th to $m$th rows of the $j$th column, and $\boldsymbol{M}(l,j)$ denotes its $(l,j)$ element. For  $i=1,\cdots,u$, let $r_{i}=\sum_{k=0}^{i}n_{k}$,  and $B_{ij} =\boldsymbol{M}(r_ {i-1}+1:r_ {i},j)$ denotes $i$th slice in the $j$th column of $\boldsymbol{M}$  , where  $n_{0}=0$, $r_{0}=0$.  Define $E_{ij}=\{\boldsymbol{M}(r_{i}+1, j),\cdots,\boldsymbol{M}(n, j)\}$, where $i=1,\cdots,u-1$ and $n=\sum_{i=1}^{u}n_{i}$.  Let $A_{L}=\{1,\cdots, L\}$ denote a set of integers from 1 to $L$, where $L=$ lcm$(n_{1},\cdots,n_{u},n)$.
Set $B=\{\boldsymbol{M}(1,j),\cdots,\boldsymbol{M}(n,j) \}$. Let $C=A_{L}\setminus B$ denote $A$ minus $B$.

It is observed that elements of each slice on an FSLHD are fixed by the construction method in Section \ref{Section 2}. There are two situations. On the one hand, some elements in an arbitrary column of an FSLH from different slices are exchanged, and the resulting FSLH  does not change the sliced structure. On the other hand, some elements which are used to construct an column of an FSLH are not selected in $C$,  besides, we  exchange some elements  between $B_{ij}$ and $C$, and the resulting  FSLH  still keeps the  sliced structure. It motivates us to  propose a different-slice exchange  procedure and an out-slice exchange procedure to generate more diverse neighbors of the design. By the above ways, we can more easily find the optimal design. The detailed process of the two procedures are as follows.

\textbf{The different-slice exchange procedure in the $i$th slice:} we select any element $b$ of $B_{ij}$.  Let $\rho(b)$ be a subset of $E_{ij}$ satisfying that the generated FSLH still keeps the  sliced structure by exchanging  $b$ with arbitrary  $c$ in $\rho(b)$, where $i=1,\cdots,u-1$. \

\textbf{The out-slice exchange  procedure in the $i$th slice:} the elements in $C$ are called out-slice elements in a column of the design. For the same $b$,
let $\sigma (b)$ be a subset of $C$ satisfying that the obtained  FSLH still maintains the sliced structure through exchanging $b$ with arbitrary $c$ in $\sigma(b)$, where  $i=1,\cdots,u$. Let $\tau (b)= \rho(b) \cup \sigma(b)$. In the last slice, we only consider the out-slice exchange  procedure, thus $\tau (b) = \sigma(b)$.
For a set $R$, $R_{k}$ denotes the $k$th smallest element  of $R$. Suppose that $\boldsymbol{M}_{\textrm{N}}(1:n,j)$ is  a new column generated from  $\boldsymbol{M}(1:n,j)$.  Here, for $i =1,\cdots,u$, recall that  $t^{i}=L/n_{i}$.
We provide a method to generate $\tau(b)$ in the $i$th  slice of  $\boldsymbol{M}$ by the following  steps:
\begin{enumerate}  [itemindent=2em] 
	\item[\textbf{Step 1.}] Randomly select an element $b$ in $\boldsymbol{M}(r_ {i-1}+1:r_{i},j)$.
	\item[\textbf{Step 2.}] Generate a set $R=\{ (\lceil b/t^{i}\rceil - 1)\times t^{i}+1,(\lceil b/t^{i}\rceil - 1)\times t^{i}+2,\cdots, \lceil b/t^{i}\rceil  \times t^{i}\} \backslash \{b\}$.
	\item[\textbf{Step 3.}]If $i< u$, go to \textbf{Step 4}; else, go to \textbf{Step 5}.
	\item[\textbf{Step 4.}]For $k$ from $1$ to $ t^{i}-1$, if  $R_{k}$  belongs to  $\boldsymbol{M}(r_{i}+1:n,j)$, go to \textbf{Step 5}; else, go to \textbf{Step 6}.
	\item[\textbf{Step 5.}] Generate $\boldsymbol{M}_{\textrm{N}}(1:n,j)$  by exchanging  $b$ with $R_{k}$.  If $\boldsymbol{M}_{\textrm{N}}(1:n,j)$  still satisfies Theorem \ref{Theo 1}(ii), go to \textbf{Step 7}.
	\item[\textbf{Step 6.}] Generate $\boldsymbol{M}_{\textrm{N}}(1:n,j)$  by exchanging  $b$ with  $R_{k}$.
	If $\boldsymbol{M}_{\textrm{N}}(1:n,j)$ still satisfies Theorem \ref{Theo 1}(i), go to \textbf{Step 7}.
	\item[\textbf{Step 7.}] Add $R_{k}$  to $\tau(b)$.
\end{enumerate}
\begin{figure*}[h]
	\centering
	\subfigure[ Different-slice exchange procedure]{
		\label{Fig2a}
		\includegraphics[width = 0.4\textwidth]{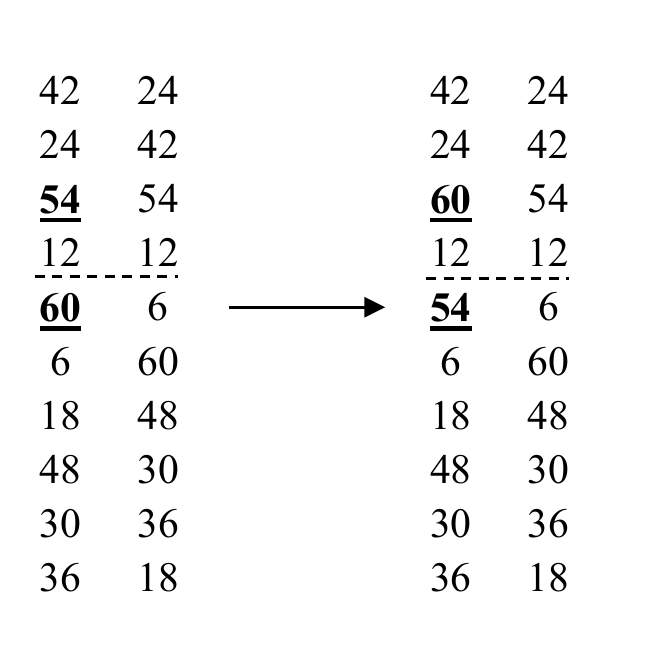}
	}
	\subfigure[Out-slice exchange procedure]{
		\label{Fig2b}
		\includegraphics[width = 0.4\textwidth]{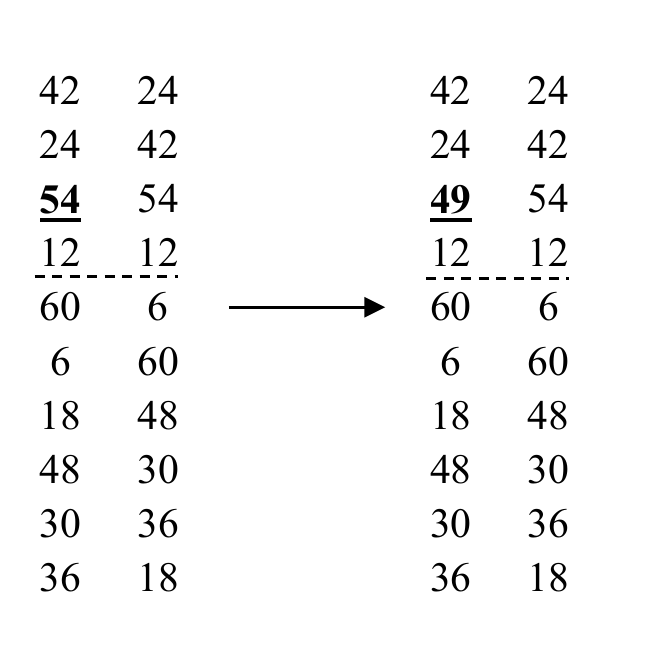}
	}
	\caption{(a): The different-slice procedure: exchange 54 in $\boldsymbol{M}(1:4,1)$ with 60 of $\tau (b)$ in $\boldsymbol{M}(5:10,1)$. (b): The out-slice procedure: replace 54 in $\boldsymbol{M}(1:4,1)$ with $ 49$ of $\tau (b)$ in  the out-slice elements.}
	\label{Fig2}
\end{figure*}
\textbf{Step 5} and \textbf{Step 6} are critical for generating $\tau(b)$. In \textbf{Step 5}, since both $b$ and $R_{k}$ are in  $\boldsymbol{M}(1:n,j)$,  $\boldsymbol{M}_{\textrm{N}}(1:n,j)$ still satisfies Theorem (i), when
we exchange $b$ with $R_{k}$.   Thus, we just guarantee that $\boldsymbol{M}_{\textrm{N}}(1:n,j)$ still satisfies Theorem \ref{Theo 1}(ii). In \textbf{Step 6}, it is clear that  changing $b$ with any element of  $R$ can guarantee that $\boldsymbol{M}_{\textrm{N}}(1:n,j)$ still satisfies Theorem \ref{Theo 1}(ii),
therefore, we only ensure that $\boldsymbol{M}_{\textrm{N}}(1:n,j)$  satisfies Theorem \ref{Theo 1}(i).

We introduce the different-slice exchange and the out-slice exchange procedures  in Figure \ref{Fig2}. For an FSLH(4,6; 2,2)$(\boldsymbol{M})$, we randomly select $b=54$ in $\boldsymbol{M}(1:4,1)$ in Figure \ref{Fig2a} , then $t^{1}=60/4=15$ and $R=\{45,46,\cdots,53,55,\cdots,60\}$. We obtain $\tau (b)=\{49,50,51,52,53,60\}$ after conducting the above steps.
In the different-slice exchange procedure, we can exchange 54 with 60 of  $\tau (b)$ in Figure \ref{Fig2a}.  In the out-slice exchange procedure,  we can replace 54 with 49 of $\tau (b)$  in Figure \ref{Fig2b}. It can be seen that the two resulting designs still keep the sliced structure.

\subsection{A sliced ESE algorithm for generating optimal FSLHDs}
\label{subsection3.3}
Researchers utilize various optimization algorithms to construct  optimal LHDs, such as  the enhanced stochastic evolutionary (ESE) algorithm \cite{Jin2016An}, the simulated annealing search algorithm \cite{Morris1995Exploratory}, the column wise-pairwise swap algorithm \cite{Ye2000Algorithmic}, the threshold accepting algorithm \cite{Fang2002Centered}, the particle swarm algorithm \cite{Chen2013Optimizing,Kennedy1995Particle}, and the genetic algorithm \cite{Liefvendahl2006A,Bates2004Formulation}. All the above algorithms can be extended to  optimize FSLHDs.  In this paper, we choose the ESE algorithm as a basic algorithm to find optimal FSLHDs. 

The ESE algorithm can quickly construct an optimal LHD in a limited calculative resource and it can also move from a locally optimal LHD.   The ESE algorithm includes double loops, i.e., an inner loop and an outer loop. The  inner loop  randomly generates neighbors of the design  by the exchange procedures and decides  whether  to  accept  them on the basis of an acceptance criterion.  The outer loop aims to adjust the threshold $T_{\textrm{h}}$ in the acceptance criterion through the performance of the inner loop, so the outer loop can control the whole optimization process. When extending the  ESE algorithm for  searching an optimal FSLHD,  we need to consider  the sliced structure of an FSLHD.  Thus,  based on the  three exchange procedures in Section \ref{subsection3.2},  we develop a sliced enhanced stochastic evolutionary (SESE) algorithm which contains double loops in \cite{Jin2016An} and the slice by slice loop  proposed in this article. Such a combined algorithm can suit the sliced structure of the FSLHD.  It is a dynamic optimization approach to optimize the FSLHD  slice by  slice.
This  algorithm can search the optimal FSLHD by minimizing the CSM. Algorithm \ref{Algorithm 1} describes  the SESE algorithm.

\textbf{The slice by slice loop:}  we start with  an initial FSLHD denoted by $\boldsymbol{D}_{0}$. When we optimize the first slice of the design, $\boldsymbol{D}_{0}$ is an initial design in the outer loop. When optimizing  the $i$th ($i \ge 2$) slice of the design, we  make $\boldsymbol{D}_{\textrm{best}}$, generated from outer loop in the $(i-1)$th slice optimization, as the initial FSLHD. It means that a new slice optimization is based on the previous slice  optimization  until the last slice.	The parameter settings of the inner loop and the outer loop have been discussed in \cite{Jin2016An}. The parameter settings are similar in  \cite{Jin2016An} for the construction method of an FSLHD.

\textbf{The inner loop:}
the iterations $P$  should be set larger for larger problems but no larger than 100. The acceptance criterion is $\phi_ {\textrm{CSM}}(\boldsymbol{D}_{\textrm{N}})-\phi_ {\textrm{CSM}}(\boldsymbol{D})\leq  T_{h} \cdot random(0,1)$, where $random(0,1)$ generates uniform numbers between 0 and 1.  According to the discussion in \cite{Jin2016An}, if the settings of $I_{1},I_{2}$, and $I_{3}$ are too large, it can appear the locally optimal design for designs with small run sizes and low efficiency for designs with large run sizes. 
Let $I_{1}=$ min $(n_{\textrm{in-slice}}/5,50)$, where $n_{\textrm{in-slice}}$ is the number of all possible neighbors of the design in within-slice exchange procedure. Let $n_{\textrm{diff-slice}}$ and $n_{\textrm{out-slice}}$ be the number of all possible neighbors of the design for the different-slice exchange procedure and the out-slice exchange procedure, respectively. According to the construction method of the FSLHD, we can clearly know that
$n_{\textrm{diff-slice}}$ and $n_{\textrm{out-slice}}$ are usually small, therefore it is reasonable  to set $I_{2}+I_{3}=$ \textrm{min}$(n_{\textrm{diff-slice}}+n_{\textrm{out-slice}},50)$.

\textbf{The outer loop:} 
The setting of $T_{\textrm{h}}$ is a small value, i.e., $T_{\textrm{h}_{0}}= 0.005\times $ (criterion value of the initial design). The  threshold $T_{\textrm{h}}$ is adjusted by an improving process and an exploration process. 
After the Inner Loop, if the search process has improvement, then go to the improving process, while if the search process has no improvement, then  go to the exploration process. We adjust $T_{\textrm{h}}$ by the same way  in \cite{Jin2016An} as follows.
In the improving process,  when $T_{\textrm{h}}$ keeps on a small value, only slightly worse design or better design will be accepted. The parameter $P$ is the number of tries in the inner loop. The  threshold $T_{\textrm{h}}$  is adjusted by the acceptance ratio
$p_{\textrm{ac}}=n_{\textrm{ac}}/P$ ($n_{\textrm{ac}}$, the number of the accepted designs) and the improvement ratio $p_{\textrm{im}}=n_{\textrm{im}}/P$  ($n_{\textrm{im}}$, the number of the improved designs).
For $flag_{\textrm{im}}=1$,   if $p_{\textrm{ac}}>0.1$ and $p_{\textrm{im}}< p_{\textrm{ac}}$, let $T_ {\textrm{h\_try}}=\beta_{1}T_{h}$, where $0< \beta_{1} <1$; if
$p_{\textrm{ac}}>0.1$ and $p_{\textrm{im}}= p_{\textrm{ac}}$,  let $T_ {\textrm{h\_try}}=T_{\textrm{h}}$; otherwise, $T_ {\textrm{h\_try}}=T_{\textrm{h}}/\beta_{1}$.  We set $\beta_{1}=0.8$, since it appears to do well in all tests. \noindent  In the exploration  process, $T_{\textrm{h}}$  is adjusted by $p_ {\textrm{ac}}$.
For $flag_{\textrm{im}}=0$,
let $T_{\textrm{h\_try}}=T_{\textrm{h}}/\beta_{2}$ and $T_{\textrm{h}}$ will be quickly increased until $p_{\textrm{ac}}>0.8$; if $p_{\textrm{ac}}>0.8$, let $T_{\textrm{h\_try}}=T_{\textrm{h}}\beta_{3}$ and $T_{\textrm{h}}$ will be quickly decreased until $p_{\textrm{ac}}<0.1$, where $0< \beta_{2},  \beta_{3}<1$. On the basis of some tests, the settings of $\beta_{2}=0.7$ and $\beta_{3}=0.9$ perform well.
Increasing rapidly $T_{\textrm{h}}$ (more worse designs can be accepted) is useful to go away from a locally optimal design.
After going away from a locally optimal design, decreasing slowly  $T_{\textrm{h}}$ helps to search better designs.  An improved design is found by repeating the exploration process, then we go into the improving process.  The $tol$ is a small fixed value, i.e., $tol=0.1$. The stopping criterion $N$  is  set to be 10 in our procedure, which is selected flexibly.
\begin{algorithm*}[H]
	\caption{The  SESE algorithm }
	\label{Algorithm 1}
	\KwIn{An initial design $\boldsymbol{D}_{0}$.}
	Initialization: $\boldsymbol{D}_{{\textrm best}} = \boldsymbol{D}_{0}$.\\
	\For{$i=1,\cdots,u$}{\textbf {Slice-by-Slice Loop}:\\
		$\boldsymbol{D}_{0}=\boldsymbol{D}_{\textrm{best}}$. \\
		\textbf {Outer Loop}:\\
		Initialization: $\boldsymbol{D}=\boldsymbol{D}_{0}$, $\boldsymbol{D}_{\textrm{best}}=\boldsymbol{D}$, $T_{\textrm{h}}=T_{\textrm{h}_{0}}$.\\
		\For{$j=1,\cdots,N$}{
			$\boldsymbol{D}_{\textrm{old}\_{\textrm{best}}}=\boldsymbol{D}_{\textrm{best}}$, \\
			$n_{\textrm{ac}}=0$, $n_{\textrm{im}}=0$.\\
			\textbf {Inner Loop}:\\
			\For{$k=1,\cdots,P$}{
				In the $i$th slice of the design, randomly choose $I_{1}$, $I_{2}$ and $I_{3}$ neighbors of  the design  by \indent the within-slice exchange, the different-slice exchange, and the out-slice exchange procedures within column
				$( k \mod  q)+1 $ , respectively. Select the best design $\boldsymbol{D}_{\textrm{N}}$ from $(I_{1}$ +$I_{2}$ +$I_{3})$  designs.\\
				\If{$\phi_ {{\textrm the  CSM}}(\boldsymbol{D}_{\textrm{N}})-\phi_ {\textrm{CSM}}(\boldsymbol{D})\leq  T_{\textrm{h}} \cdot random(0,1)$}
				{
					$\boldsymbol{D}=\boldsymbol{D}_{\textrm{N}}$,\\
					$n_{\textrm{ac}}= n_{\textrm{ac}}+1$.\\
					\If{$\phi_ {\textrm{CSM}}(\boldsymbol{D}) < \phi_ {\textrm{CSM}}(\boldsymbol{D}_{\textrm{best}})$}
					{
						$\boldsymbol{D}_{\textrm{best}}=\boldsymbol{D}$, \\
						$n_{\textrm{im}}=n_{\textrm{im}}+1$.\\
					}
				}
			}
			\eIf{$\phi_ {\textrm{CSM}}(\boldsymbol{D}_{\textrm{old}\_{\textrm{best}}})-\phi_ {\textrm{CSM}}(\boldsymbol{D}_{\textrm{best}})> tol$}
			{
				$flag_{\textrm{im}}=1$.\\
			}
			{ $flag_{\textrm{im}}=0$.\\
			}
			Update $T_{\textrm{h}}$ according to $flag_{\textrm{im}}$, $n_{\textrm{ac}}$, $n_{\textrm{im}}$.
		}
	}
	\KwOut{$\boldsymbol{D}_{\textrm{best}}$.}
\end{algorithm*}

\subsection{Efficient two-part algorithm for generating space-filing FSLHDs}
\label{subsection3.4}
For an FSLHD with $n$ runs and $q$ factors, when $n$ and $q$ are small, the SESE algorithm is more efficient and provides much better resulting designs.	However,  if $n$ and $q$ are getting larger,  the convergence of  the SESE algorithm  may be slow because of the large  number of  neighbors of the design.  In this subsection, we consider  a  similar strategy which is broadly applied in \cite{Ba2015Optimal,Chen2017Flexible} to avoid the poor space-filling designs and improve the efficiency when $n$ and $q$ are large.

We first give the strategy for our proposed design as follows: for an FSLHD$(n_ {1},\cdots,n_ {u};u,q)$ and $n=\sum_{i=1}^{u}n_{i}$,
the $q$-dimensional input region in the $i$th slice of FSLHD is partitioned  into $n_ {i}^{q}$ cells through the $\underbrace{n_ {i}\times,\cdots,\times n_{i}}_{q}$ coarser grid ($i=1,\cdots,u$). Since  run sizes $n_{i}$ of each slice are different, the number  $n_ {i}^{q}$  of divided cells is different. 
It is possible that some of $n$ design points sampled from the  $n_ {i}^{q}$ cells can fall into the same cell. If $n_{i}^{q} > n$, we need to  avoid design points  falling into in the same cells and ensure the design still an FSLHD. 

We give a detailed process of the above strategy. Let $\bm{1}\{\cdot\}$  denote the indicator function. For an $n\times q$  matrix $\boldsymbol{A}=[\boldsymbol{a}_{1},\cdots,\boldsymbol{a}_{n}]^{\textrm{T}}$, denote
\begin{equation}
\label{indicator function}
P(\boldsymbol{A})=\sum_{1\leq i < j \leq n}\bm{1}\{d(\boldsymbol{a}_{i}, \boldsymbol{a}_{j})=0\},
\end{equation}
where $\bm{1}\{d(\boldsymbol{a}_{i}, \boldsymbol{a}_{j})=0\}=1$ if
$d(\boldsymbol{a}_{i},\boldsymbol{a}_{j})=0$ is true and $\bm{1}\{d(\boldsymbol{a}_{i}, \boldsymbol{a}_{j})=0\}=0$ otherwise. 
It is clear that some rows of matrix $A$ are the same if $P(\boldsymbol{A})>0$. We call the same rows as repeating rows which fall into the same cell. We can find repeating rows of a design  by (\ref{indicator function}).
For FSLH($n_{1},\cdots,n_{u};u,q$) ($\boldsymbol{{M}}$), recall that $t^{i}=$lcm $(n_{1},\cdots,n_{u},n)/n_{i}$, for $i=,\dots,u-1$. 
Let $\boldsymbol{M}^{i}=\lceil \boldsymbol{M}/t_{i}\rceil$. If $n_{i}^{q} > n$ and $P(\boldsymbol{M}^{i})=\sum_{1\leq i < j <n}\textbf{1}\{d( \boldsymbol{a}_{i}, \boldsymbol{a}_{j})=0\}>0$, then  the matrix has repeating rows.

Let us look at the following example of a design matrix FSLH (4,6;2,2) ($\boldsymbol{M}$)
$$\boldsymbol{M}=\left(
\begin{array}{cccccccccc}
54 & 12 & 24 & 42 & 60 &  30  & 6 &18& 48& 36\\
54 & 42 & 12 & 24 & 18 &  6  & 36 & 48 & 60& 30
\end{array}
\right)^{T}.$$
By (\ref{matrix FSLH1}) and (\ref{matrix FSLH2}), both $\boldsymbol{M}^{1}=\lceil \boldsymbol{M}/15 \rceil$ and $\boldsymbol{M}^{2}=\lceil \boldsymbol{M}/10 \rceil$  have repeating rows, which indicates that $P( \boldsymbol{M}^{1} ) =4>0$ and $P(\boldsymbol{M}^{2}) =1>0$. The FSLH corresponding to the  design under different divided cells is depicted in Figure \ref{fig3a} and Figure \ref{fig3b}, respectively.
The design points of repeating rows fall into the same cell (filled with blue).
\begin{figure}[H]
	\label{fig3}
	\centering 
	\subfigure[]{
		\begin{minipage}[t]{0.45\linewidth}
			\centering
			\includegraphics[width=1.22\textwidth]{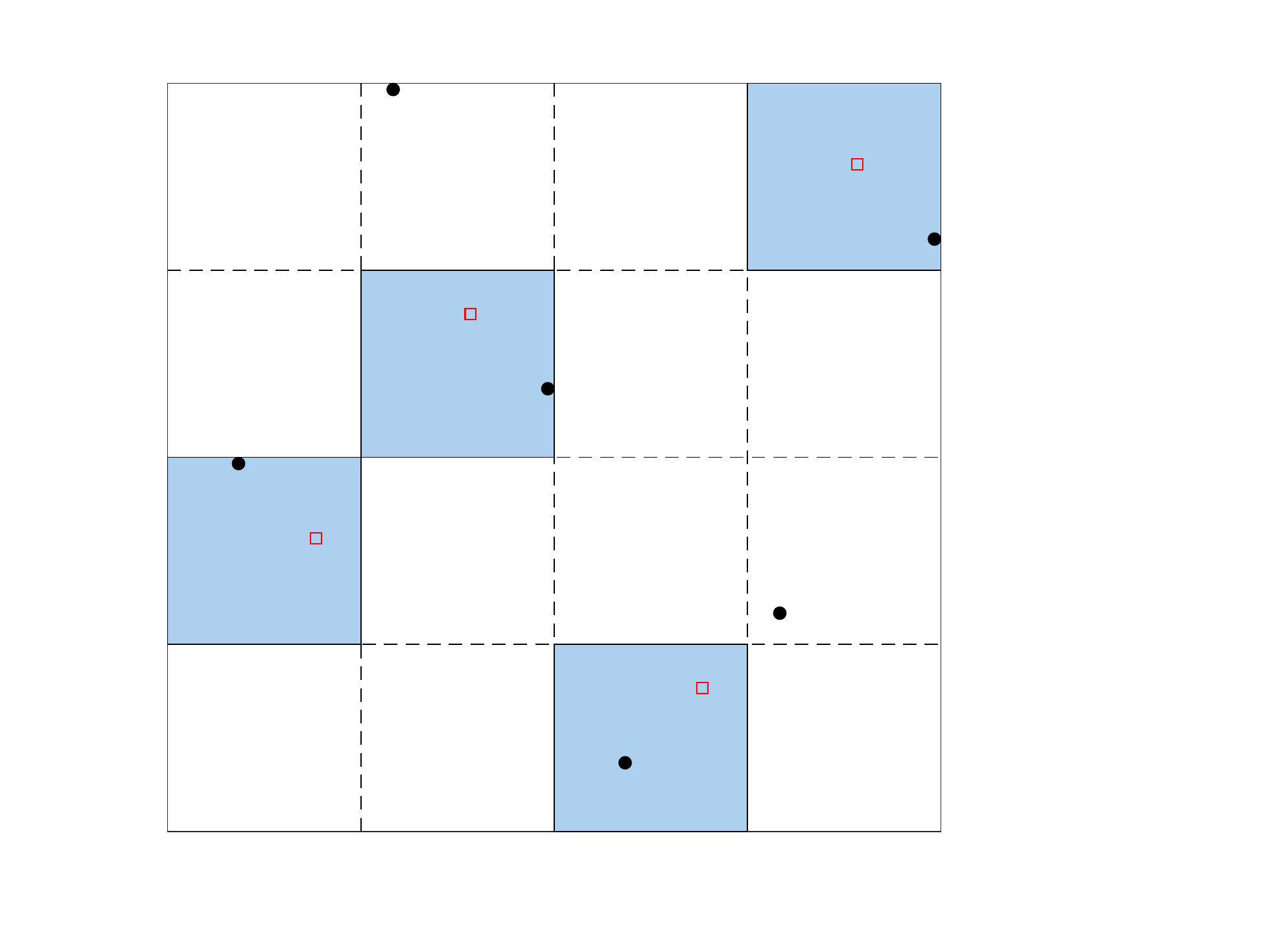} 
			\label{fig3a}
	\end{minipage}}
	\subfigure[]{
		\begin{minipage}[t]{0.45\linewidth}
			\centering
			\centering \includegraphics[width=1.2\textwidth]{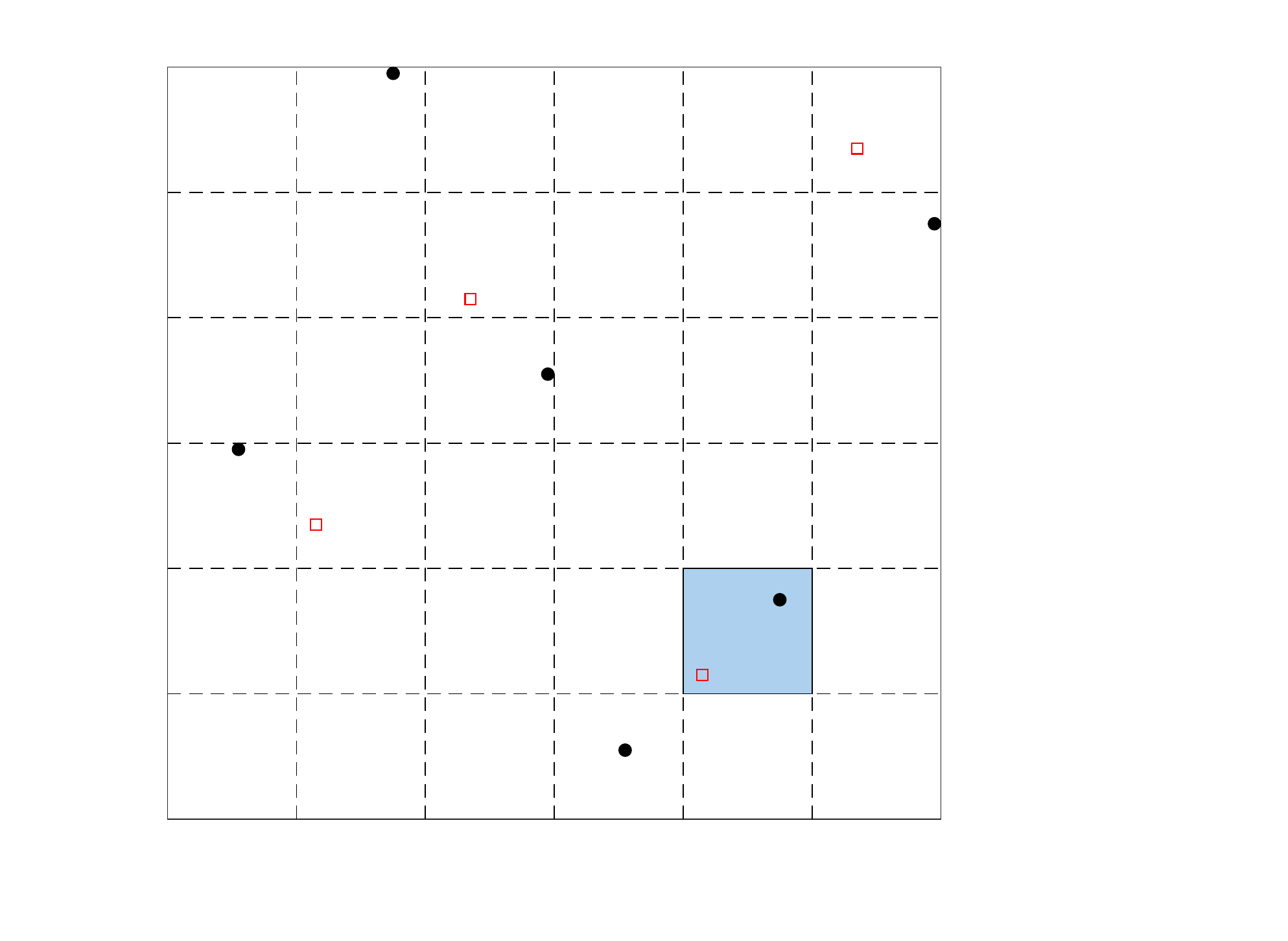}
			\label{fig3b}
		\end{minipage}
	}
	\centering
	\caption{ A poor design with some repeating rows. (a) :  The 2-dimensional input region  is divided into $4\times4$ cells, and some repeating rows lie in the same cell.  (b) : The 2-dimensional input region  is divided into $6\times6$ cells, and some repeating rows lie in the same cell.}
\end{figure}

\begin{equation}
\label{matrix FSLH1}
\boldsymbol{M}^{1}=\left(
\begin{array}{cccccccccc}
4  &3	&1	&2	&2	&1	&3	&4	&4	&2\\
4	&1&2	&3	&4	&2&1	&2	&4	&3
\end{array}
\right)^{T},
\end{equation}

\begin{equation}
\label{matrix FSLH2}
\boldsymbol{M}^{2}=\left(
\begin{array}{cccccccccc}
6  &5	&2&3	&2	&1	&4	&5	&6&3\\
6	&2&3&5	&6	&3&1	&2	&5	&4
\end{array}
\right)^{T}.
\end{equation}

To make the design with better space-filling properties, we consider to put all the points  into the different cells.
Therefore, we can select randomly a column of  the repeating rows, and conduct a within-slice exchange procedure in the randomly chosen column of the same slice, until $P(\boldsymbol{M}^{1})=0$ and $P(\boldsymbol{M}^{2})=0$. The resulting design are shown in Figure \ref{fig4a} and Figure \ref{fig4b}, respectively, in which  all the points fall into the different cells. In summary, the above strategy can quickly eliminate the undesirable designs that contain repeating rows.
\begin{figure}[H]
	\label{fig4}
	\centering
	\subfigure[]{
		\begin{minipage}[t]{0.45\linewidth}
			\centering
			\includegraphics[width=1.22\textwidth]{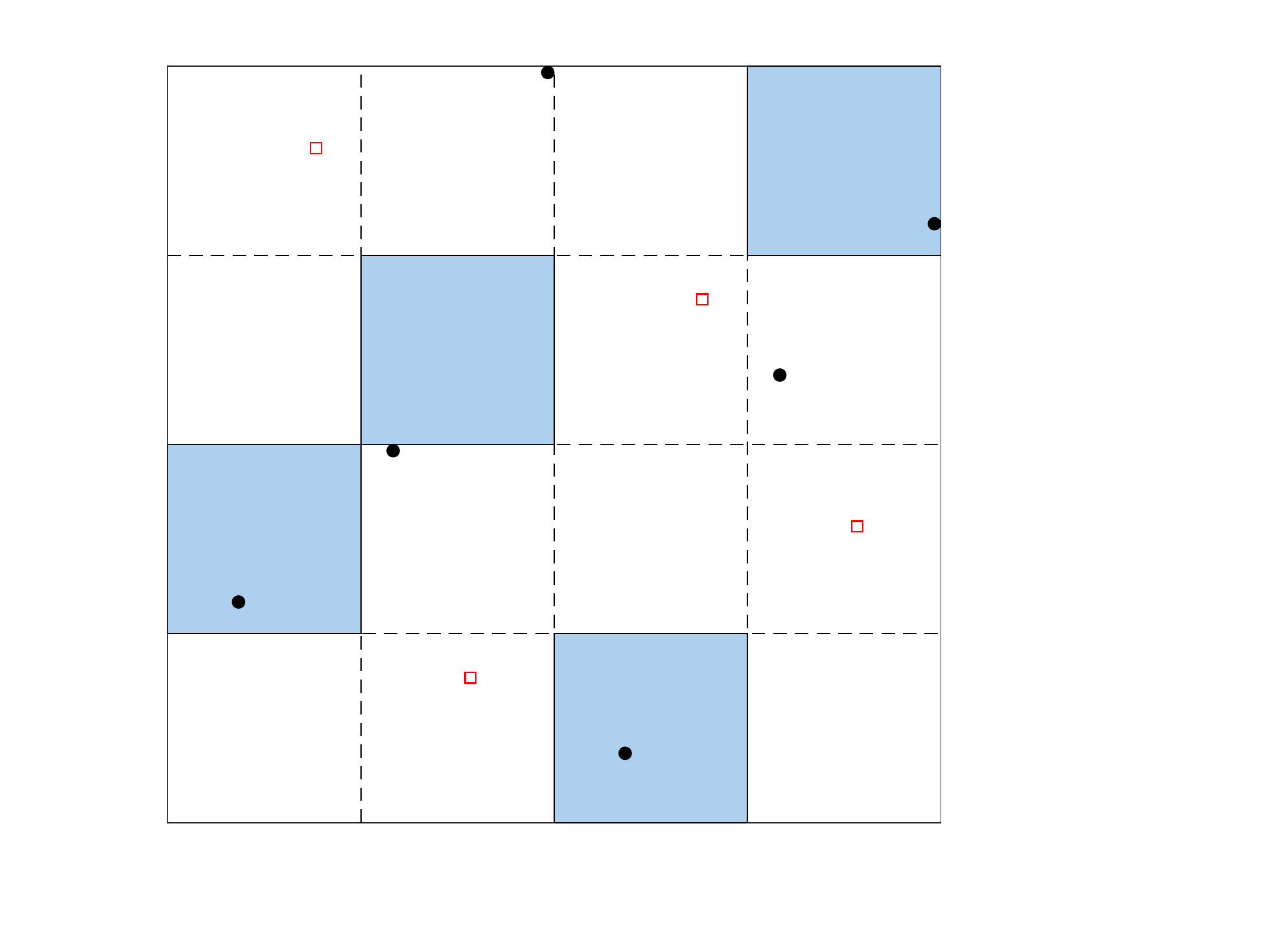}
			\label{fig4a}
	\end{minipage}}
	\subfigure[]{
		\begin{minipage}[t]{0.45\linewidth}
			\centering
			\centering \includegraphics[width=1.2\textwidth]{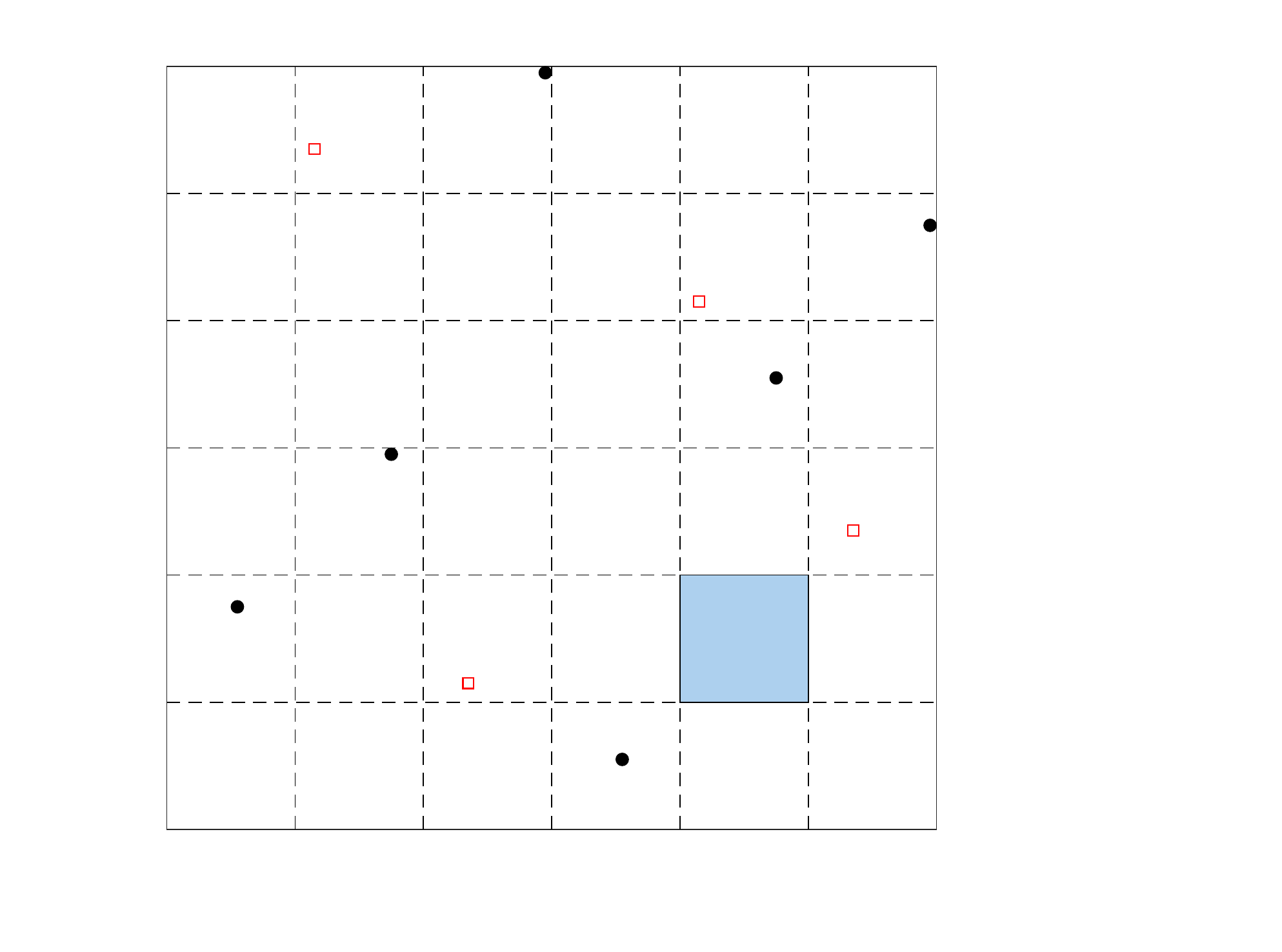}
			\label{fig4b}
		\end{minipage}
	}
	\centering
	\caption{A resulting design with design points spread out. (a) : The $n$ design points fall into different cells in $4\times4$ grid. (b) : The $n$ design points fall  into different cells in $6\times6$ grid.}
\end{figure}

Given an FSLHD with large $n$ runs and $q$ factors, we develop an efficient two-part algorithm for finding the space-filling FSLHDs  based on the above strategy. Without loss of generality, assume  $n_{1},\cdots,n_{u}$ with $n_{1}\leq n_{2}\leq ,\cdots,\leq n_{u}$. Recall that
$\boldsymbol{D}_{\textrm{N}}$  denotes a neighbor of FSLHD($\boldsymbol{D}$) and  $\boldsymbol{M}_{\textrm{N}}$  denotes a neighbor of FSLH($\boldsymbol{M}$).
This algorithm is provided as follows:\\
\textbf{Part-I algorithm}

The Part-I algorithm is useful for speeding up  by removing some undesirable designs from neighbors of the design. It starts with an initial FSLH ($n_{1},\cdots,n_{u};u,q$)$(\boldsymbol{M}_{0})$. According to the run sizes of the design, it can be stopped by some flexible stopping criterions. In our proposed algorithm, when 100 iterations have been operated, we stop the program.
The  algorithm  is given below:
%\vspace{6pt} 
\begin{enumerate}  [itemindent=2em] 
	\item[\textbf{Step 1.}] Let $\boldsymbol{M}=\boldsymbol{M}_{0}$, and set the index $i=1$.
	\item[\textbf{Step 2.}] If  $P(\lceil \boldsymbol{M}/t_{i} \rceil)=0$, compute $\phi_ {\textrm {CSM}}(\boldsymbol{D})$,  go to \textbf{Step 5}.
	\item[\textbf{Step 3.}] If $n_{i}^{q}>n$, randomly choose a repeating row of  $\lceil \boldsymbol{M}/t_{i}\rceil$, and randomly choose  another row in the same slice.  We exchange  two elements which corresponds to a randomly selected column of the two rows.  Generate an $\boldsymbol{M}_{\textrm{N}}$;  else,  go  to  \textbf{Step 5}.
	\item[\textbf{Step 4.}] If $P(\lceil \boldsymbol{M}_{\textrm{N}}/t_{i} \rceil)<P(\lceil \boldsymbol{M}/t_{i} \rceil)$, $\boldsymbol{M}=\boldsymbol{M}_{\textrm{N}}$,  go back to  \textbf{Step 2};  else, go back to   \textbf{Step 3}.
	\item[\textbf{Step 5.}]Under the condition of $P(\lceil \boldsymbol{M}_{\textrm{N}}/t_{i} \rceil) =0$, generate an  $\boldsymbol{M}_{\textrm{N}}$ by the within-slice  procedure in the $i$th slice of $\boldsymbol{M}$,  then calculate $\phi_ {\textrm{CSM}}(\boldsymbol{D}_{\textrm{N}})$.
	\item[\textbf{Step 6.}]  If $\phi_ {\textrm{CSM}}(\boldsymbol{D}_{\textrm{N}}) < \phi_ {\textrm{CSM}}(\boldsymbol{D})$, then replace $ \boldsymbol{M}$ by $\boldsymbol{M}_{\textrm{N}}$;  else,  go back to   \textbf{Step 3}.
	\item[\textbf{Step 7.}] Repeat  \textbf{Step 4} and  \textbf{Step 5} until meeting the stopping criterion.
	\item[\textbf{Step 8.}] Update  $ i=i+1$, if $i < u$,  go to \textbf{Step 2}; else,  output $\boldsymbol{M}_{\textrm{best}}$ = $\boldsymbol{M}$.
\end{enumerate}

\noindent\textbf{Part-II algorithm}

We take  $\boldsymbol{M}_{\textrm{best}}$ from the Part-I  algorithm as an initial design in the Part-II algorithm.
We generate a  neighbor of FSLHD based on the different-slice or the out-slice exchange procedures  in the Part-II algorithm.  For $i=1,\cdots,u$,  if  $q$ is large and $n_{i}^{q}>>n$, then the $n$ design points is very sparse by the  Part-I  algorithm,  consequently, the Part-II  algorithm brings smaller  effect for the  space-filing properties of the design $\boldsymbol{D}$. Therefore, in this case,  the Part-I  algorithm is more important, and we can skip the Part-II algorithm and focus on the Part-I algorithm. We also can stop the running of Part-II algorithm  when the repeating times arrive 100.

\begin{enumerate}  [itemindent=2em] 
	\item[\textbf{Step 1.}] Let $\boldsymbol{M}=\boldsymbol{M}_{\textrm{best}}$, and set the index $i=1$.
	\item[\textbf{Step 2.}] In the $i$th slice of $\boldsymbol{M}$, generate an $\boldsymbol{M}_{\textrm{N}}$ by the different-slice or the out-slice exchange procedures under the condition of $P(\lceil \boldsymbol{M}_{\textrm{N}}/t_{i} \rceil) =0$.
	\item[\textbf{Step 3.}] If $\phi_ {\textrm{CSM}}(\boldsymbol{D}_{\textrm{N}}) < \phi_ {\textrm{CSM}}(\boldsymbol{D})$,  replace $ \boldsymbol{M}$ by $\boldsymbol{M}_{\textrm{N}}$.
	\item[\textbf{Step 4.}] Repeat \textbf{Step 2} and  \textbf{Step 3} until meeting the stopping criterion.
	\item[\textbf{Step 5.}]Update $i=i+1$, if $i<u$,  go to \textbf{Step 2}; else,  output $\boldsymbol{M}_{\textrm{best}}$ = $\boldsymbol{M}$.
\end{enumerate}
\section{Simulation results}
\label{Simulation results}
In this section, the first example illustrates that  the SESE  algorithm has good properties. In our second example, for the design with large runs and factors, we give some comparative studies, which show the efficient two-part algorithm with desirable performance. In these examples, we select the combined space-filling measurement (\ref{CSM}). For simplicity, we only consider any column of FSLHDs with all in (\ref{eq1a}) being 1/2 when updating (\ref{CSM}) in our proposed algorithm.

\subsection{Example 1}
As depicted in Figure \ref{figure5a}, we randomly generate an initial design 
FSLHD($4,8,12$; 3,2) with optimal univariate uniformity.  It is clear that the space-filling property is poor for the whole design and for each slice of the design. Based on the combined space-filling measurement  $\phi_{\textrm{CSM}}$ ($t=50$ ) in (\ref{CSM}), we improve the space-filling property of the design by the SESE algorithm ($P= 20$). The initial design with $\phi_{\textrm{CSM}}= 14.4740$  is showed in Figure \ref{figure5a}. After
operating the SESE algorithm, the  resulting design with  $\phi_{\textrm{CSM}}=5.7958$  in Figure \ref{figure5b} has good space-filling property over the experiment region. 
\begin{figure*}[h]
	\centering
	\subfigure[$\phi_{\textrm{CSM}}= 14.4223$]{
		\begin{minipage}[t]{0.45\linewidth}
			\centering
			\includegraphics[width=1.2\textwidth]{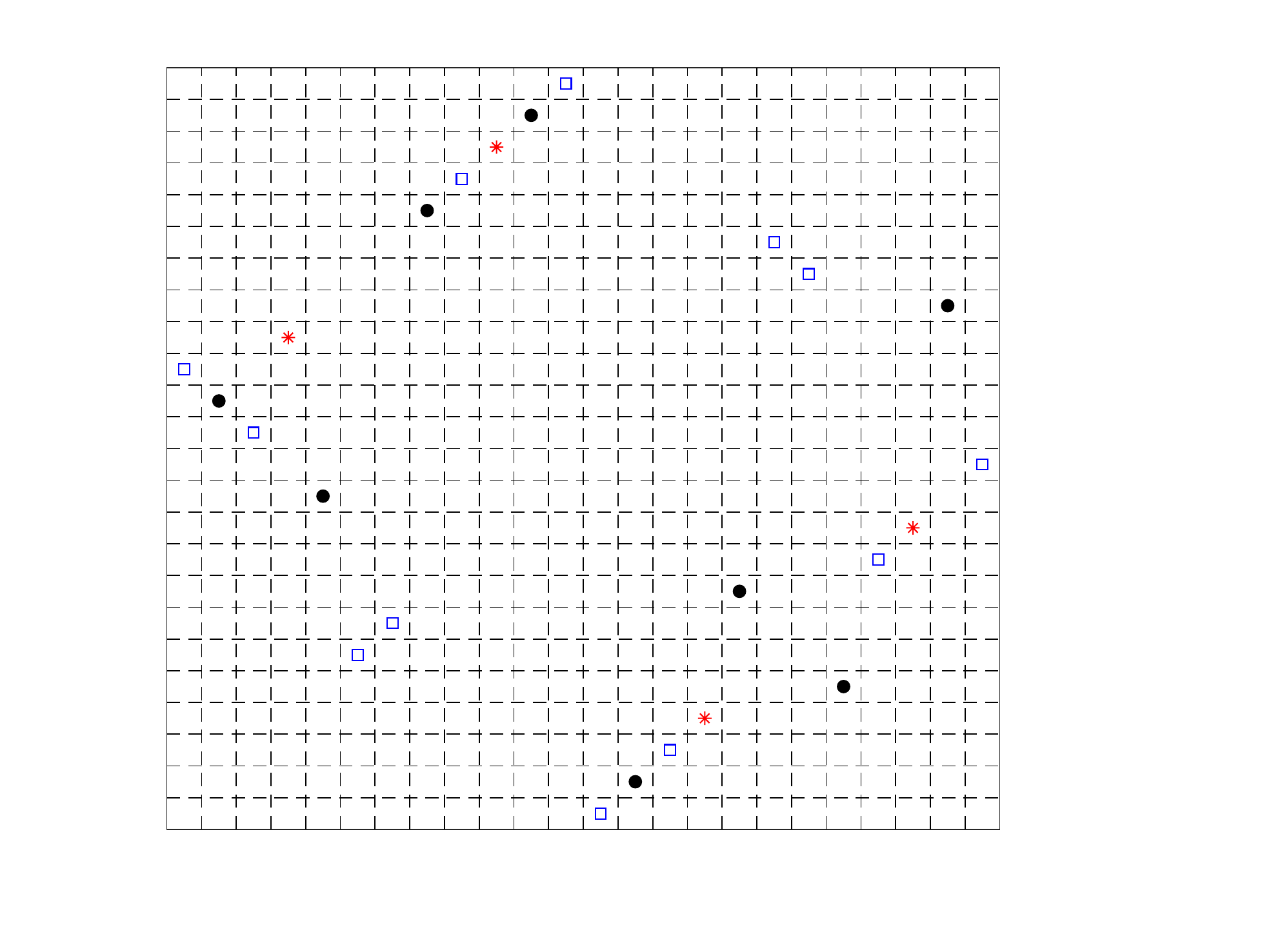}
			\label{figure5a}
		\end{minipage}
	}
	\subfigure[$\phi_{\textrm{\textrm{CSM}}}=5.6844$]{
		\begin{minipage}[t]{0.45\linewidth}
			\centering
			\includegraphics[width=1.25\textwidth]{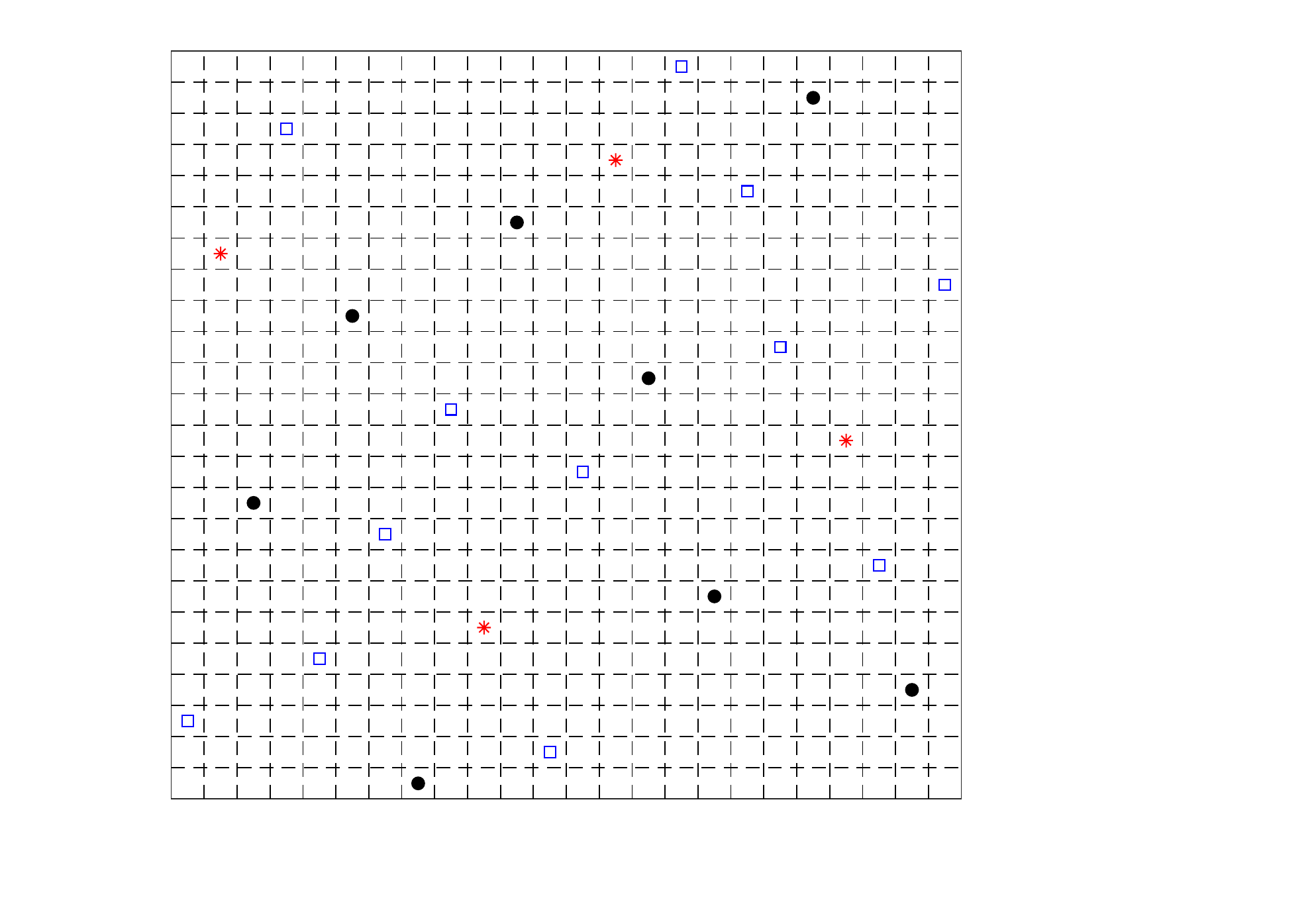}
			\label{figure5b}
		\end{minipage}
	}
	\centering
	\caption{ Optimization results for finding optimal FSLHD.
		$(a)$: The initial FSLHD$(4,8,12;3,2)$ in Example 1, different types of points denote difference three slices, respectively. $(b)$: The optimization results of  FSLHD$(4,8,12;3,2)$  after using the SESE algorithm.}
	\vspace{-0.2cm}
	\label{figure3}
\end{figure*}

For comparison, we  randomly  generate  FSLHDs by the method in  Section \ref{Section 2} for 100000 times and calculate the corresponding values of  $\phi_{\textrm{CSM}}$. The resulting FSLHDs with good space-filling properties account for a small portion of 100000 FSLHDs.
The smallest value of  $\phi_ {\textrm{CSM}}$  from the 100000  FSLHDs is 6.8387, while the  value of  $\phi_ {\textrm{CSM}}$  in Figure \ref{figure5b} is 5.7958.  The values between 6.8387 and 8 of $\phi_ {\textrm{CSM}}$ account for 0.22 percent of all  $\phi_ {\textrm{CSM}}$ values  from the 100000  FSLHDs. It can be seen that the SESE algorithm is useful to improve the space-filling property of the whole design and each slice of the initial design.

\subsection{Example 2}
To show the  good  performance of the two-part algorithm for design with large runs and factors, we compare  its performance with the SESE algorithm. We repeat each algorithm for 100 times with a random initial design FSLHD in Table \ref{table1}  (SD, standard deviation  and AT, average time ).  In the SESE algorithm, we set stopping rules $P=30$ for FSLHD$(15,30; 2,2)$  and $P=40$ for  FSLHD$(5,10,15,30;4,6)$. Conclusion can be obtained from Table \ref{table1} as follows:
\begin{enumerate}
	\item[\textbf{(i)}]  The average time of the operation shows that the two-part algorithm has higher efficiency than the SESE algorithm.
	\item[\textbf{(ii)}]  For FSLHD$(5,10,15,30;4,6)$, since  $n_{i}^{q}>>n$ with $i=1\cdots4$, the $\phi_{\textrm{CSM}}$   values of the resulting FSLHD  from  Part-I algorithm are desirable when compared with those values from two-part algorithm. However, the results of  Part-I algorithm for FSLHD$(15,30; 2,2)$ are not good enough. Therefore, if  $q$ is large and $n_{i}^{q}>>n$, we need not to run the Part-II algorithm.
	\item[\textbf{(iii)}] Based on  the  $\phi_{\textrm{CSM}}$ values of the resulting FSLHDs, we can see that the $\phi_ {\textrm{CSM}}$ values are close to each other.
	It can be concluded that both the two-part algorithm and the SESE  algorithm are stable and do not heavily rely on the initial design.
\end{enumerate}
\begin{table*}[htbp]
	\centering
	\caption{Performance of the efficient two-part algorithm  for repeating 100 times}
	\begin{tabular}{ccccccc}
		\toprule
		\textbf{Algorithm} & \textbf{Design}& \textbf{Min}	& \textbf{Mean} & \textbf{Max} & \textbf{SD} &\textbf{AT} \\
		\midrule
		SESE                   &   FSLHD$(15,30;2,2)$             & 7.8674    &  8.3100      & 8.7988    & 0.0172   & 406 seconds \\
		Part-I                  &   FSLHD$(15,30; 2,2)$            & 9.0118     &  10.3349    & 14.5465   & 0.5958   & 12   seconds  \\
		Part-I +Part-II      &   FSLHD$(15,30;2,2)$             &8.2610     & 9.1712        & 11.2161     & 0.1276   &  18  seconds\\
		SESE                   &     FSLHD$(5,10,15,30;4,6)$    & 1.8614    &   2.0823     & 2.4874    & 0.0103   & 1110 seconds  \\
		Part-I                  &    FSLHD$(5,10,15,30;4,6)$	 & 2.0474	  & 2.2545      & 2.8380    & 0.0122  &  51 seconds \\
		Part-I +Part-II	     &    FSLHD$(5,10,15,30;4,6)$     & 1.9041     & 2.2424      & 2.0394    & 0.0031  &  71 seconds\\
		\bottomrule
	\end{tabular}
	\label{table1} 
\end{table*}
By comparison, the resulting designs is better after using  SESE algorithm. However, for generating space-filing FSLHDs with large runs and factors as well as considering the cost of time, the two-part algorithm is preferable.

\section{Discussion the methods for evaluating the combined space-filling measurement}
\label{Remarks}
Recall that $\boldsymbol{D}$ is the design matrix of an ${\rm FSLHD}(n_{1},\ldots,n_{u};u, q)$.  Since we generate a neighbor of design by exchanging two elements in one column of $\boldsymbol{D}$, we do not need to recalculate all the inter-site distances when we update $\phi_ {t}(\boldsymbol{D})$ or $\phi_{t}(\boldsymbol{D}^{(i)})$. The calculative efficiency of optimality criteria for the LHD has been discussed in \cite{Jin2016An}. Here, based on above three exchange procedures for the FSLHD, we give updating expressions of $\phi_ {\textrm{CSM}}(\boldsymbol{D})$ using the previous $\phi_{t}(\boldsymbol{D})$ and  $\phi_{t}(\boldsymbol{D}^{(i)})$ for our proposed algorithm.

For the design matrix $\boldsymbol{D}=(x_{ij})_{n \times q}$ with $n$ design points $\{\boldsymbol{x}_{1},\cdots,\boldsymbol{x}_{n}\}$, we exchange $x_{rk}$ and $x_{sk}$ in the $k$th column of the design. Let $d(\cdot,\cdot)$ be the inter-site distance before exchanging.
Let $v \neq r,s$, $1 \leq v \leq n$, as defined in (\ref{inter-point distance}), the new related inter-site distance of the two design points $\boldsymbol{x}_{r}$ and $\boldsymbol{x}_{s}$ should be updated:
\begin{equation*}
d'(\boldsymbol{x}_{r},\boldsymbol{x}_{v})= ((d(\boldsymbol{x}_{r},\boldsymbol{x}_{v}))^{m}+ h(r,s,k,v))^{1/m},
\end{equation*}
\begin{equation*}
d'(\boldsymbol{x}_{s},\boldsymbol{x}_{v})= ((d(\boldsymbol{x}_{r},\boldsymbol{x}_{v}))^{m}- h(r,s,k,v))^{1/m},
\end{equation*}
where $h(r,s,k,v)=|x_{sk}-x_{vk}|^{m}-|x_{rk}-x_{vk}|^{m}$ and the other inter-site distances are unchanged. We give a new $\phi_ {\textrm{CSM}}'(\boldsymbol{D})$ based on the previous
$\phi_{t}(\boldsymbol{D})$ and  $\phi_{t}(\boldsymbol{D}^{(i)})$ as follows:
\begin{equation*}
\phi_ {\textrm{CSM}}'(\boldsymbol{D})=w\phi_ {t}'(\boldsymbol{D})+(1-w)\left(\sum_{i=1}^{u}\lambda_{i}\phi_ {t}'(\boldsymbol{D}^{(i)})\right).
\end{equation*}
For three different procedures, the values of $\phi_ {t}'(\boldsymbol{D})$ and $\phi_{t}'(\boldsymbol{D}^{(i)})$ are determined as follow:\\
(i) \textbf{Within-slice exchange procedure.} For $e \in \{1,\cdots,u\}$,
suppose that the design points $\boldsymbol{x}_{r}$ and $\boldsymbol{x}_{s}$ are in the $e$th slice. Let $n_{0}=0$. Then we have $r,s\in J_{e}=\{\sum_{l=0}^{{e-1}}n_{l}+1,\cdots,\sum_{l=0}^{{e}}n_{e}\}$ and
\begin{equation*}
\begin{split}
\phi_ {t}'(\boldsymbol{D})=& \left( (\phi_ {t}(\boldsymbol{D}))^{t}+\sum_{1\leq v\leq n,v\neq r,s} \left(d^{'}(\boldsymbol{x}_{r},\boldsymbol{x}_{v})^{-t}-d(\boldsymbol{x}_{r},\boldsymbol{x}_{v})^{-t}\right)\right. \\
&\left. +\sum_{1\leq v\leq n,v\neq r,s}\left(d^{'}(\boldsymbol{x}_{s},\boldsymbol{x}_{v})^{-t}-d(\boldsymbol{x}_{s},\boldsymbol{x}_{v})^{-t}\right)  \right)^{1/t},
\end{split}
\end{equation*}

\begin{equation*}
\phi_ {t}'(\boldsymbol{D}^{(i)})=\\
\left\{
\begin{aligned}
&\phi_{t}(\boldsymbol{D}^{(i)}),  \qquad \qquad \qquad \qquad \qquad \qquad \qquad \qquad \qquad \qquad \ \ \text{$if$ $i\neq e$};\\
&\left( \left(\phi_ {t}(\boldsymbol{D}^{(i)})\right)^{t}+\sum_{v\in J_{e},v\neq r,s} (d'(\boldsymbol{x}_{r},\boldsymbol{x}_{v})^{-t}-d(\boldsymbol{x}_{r},\boldsymbol{x}_{v})^{-t})\right. \\
&\left. \qquad+\sum_{v\in J_{e},v\neq r,s} (d'(\boldsymbol{x}_{s},\boldsymbol{x}_{v})^{-t}-d(\boldsymbol{x}_{s},\boldsymbol{x}_{v})^{-t}\right)^{1/t},
\qquad \qquad  \text{$if$ $i=e$}.
\end{aligned}
\right.
\end{equation*}
(ii) \textbf{Different-slice  exchange procedure.}
For $e,e' \in \{1,\cdots,u\}$,
suppose that the design points $\boldsymbol{x}_{r}$ are in the $e$th slice and $\boldsymbol{x}_{s}$ in the $e'$th slice. Let $n_{0}=0$. Then we have $r\in J_{e}=\{\sum_{l=0}^{{e-1}}n_{l}+1,\cdots,\sum_{l=0}^{{e}}n_{e}\}$, $s\in J_{e'}=\{\sum_{l=0}^{{e'-1}}n_{l}+1,\cdots,\sum_{l=0}^{{e'}}n_{e'}\}$ and

\begin{equation*}
\begin{split}
\phi_ {t}'(\boldsymbol{D})= &\left((\phi_ {t}(\boldsymbol{D}))^{t}+\sum_{1\leq v\leq n,v\neq r,s} (d^{'}(\boldsymbol{x}_{r},\boldsymbol{x}_{v})^{-t}-d(\boldsymbol{x}_{r},\boldsymbol{x}_{v})^{-t})\right.\\
&\left. +\sum_{1\leq v\leq n,v\neq r,s}(d^{'}(\boldsymbol{x}_{s},\boldsymbol{x}_{v})^{-t}-d(\boldsymbol{x}_{s},\boldsymbol{x}_{v})^{-t})\right)^{1/t},
\end{split}
\end{equation*}

\begin{equation*}
\phi_ {t}'(\boldsymbol{D}^{(i)})=
\left\{
\begin{split}
&\phi_{t}(\boldsymbol{D}^{(i)}),      &\text{$if$ $i\neq e,e'$};\\ 
& \left( \left( \phi_ {t}(\boldsymbol{D}^{(i)})\right)^{t} + \sum_{v\in J_{e},v\neq r} \left(d'(\boldsymbol{x}_{r},\boldsymbol{x}_{v})^{-t}-d(\boldsymbol{x}_{r},\boldsymbol{x}_{v})^{-t}\right)\right)^{1/t},& \text{$if$ $i=e$};\\
&\left( \left(\phi_ {t}(\boldsymbol{D}^{(i)})\right)^{t}+\sum_{v\in J_{e'},v\neq s} \left(d'(\boldsymbol{x}_{s},\boldsymbol{x}_{v})^{-t}-d(\boldsymbol{x}_{s},\boldsymbol{x}_{v})^{-t}\right) \right)^{1/t}, &\text{$if$ $i=e'$}.
\end{split}
\right.
\end{equation*}
(iii) \textbf{Out-slice exchange procedure.}
For $e\in \{1,\cdots,u\}$, suppose that the element $\boldsymbol{x}_{rk}$ are in the $e$th slice and the element $x_{sk}'\in (0,1)$  are in the out slice. Let $n_{0}=0$.
Then we have $r\in J_{e}=\{\sum_{l=0}^{{e-1}}n_{l}+1,\cdots,\sum_{l=0}^{{e}}n_{e}\}$, $h(r,s,k,v)=|x_{sk}'-x_{vk}|^{m}-|x_{rk}-x_{vk}|^{m}$, $v\neq r$ and

\begin{equation*}
\begin{split}
\phi_ {t}'(\boldsymbol{D})= &\left(\left(\phi_ {t}(\boldsymbol{D})\right)^{t}+\sum_{1\leq v\leq n,v\neq r} \left(d^{'}(\boldsymbol{x}_{r},\boldsymbol{x}_{v})^{-t}-d(\boldsymbol{x}_{r},\boldsymbol{x}_{v})^{-t}\right)\right)^{1/t},
\end{split}
\end{equation*}

\begin{equation*}
\phi_ {t}'(\boldsymbol{D}^{(i)})=\left\{
\begin{split}
&\phi_{t}(\boldsymbol{D}^{(i)}), &\text{$if$ $i\neq e$};\\
&\left(\left(\phi_ {t}(\boldsymbol{D}^{(i)})\right)^{t}+\sum_{v\in J_{e},v\neq r}\left(d'(\boldsymbol{x}_{r},\boldsymbol{x}_{v})^{-t}-d(\boldsymbol{x}_{r},\boldsymbol{x}_{v})^{-t}\right)\right)^{1/t},& \text{$if$ $i=e$}.
\end{split}
\right.
\end{equation*}

Through the above description of the updating  formulas, we can improve the  efficiency of re-evaluating  $\phi_ {\textrm{CSM}}(\boldsymbol{D})$ for our proposed algorithm.
\section{Conclusions}
\label{Conclusions}
In this article, we propose a method to construct SLHDs with arbitrary run sizes.
Based on such designs, we give an SESE  algorithm to search  the optimal FSLHDs.
Moreover, we provide an efficient two-part algorithm to improve the optimization efficiency in generating the space-filling FSLHDs with large runs  and factors.
We believe that FSLHDs with optimal univariate uniformity and good space-filling properties are more widely used in computer experiments. Orthogonality is also an appealing feature for SLHDs. Orthogonal SLHDs are constructed  in \cite{yang2013construction, huang2014construction,cao2015construction}, however, orthogonality does not ensure a good space-filling property.  In the future, we will study the construction of an orthogonal-maximin SLHD with slices of arbitrary run sizes. Such a design have both orthogonality and space-filling property.

\end{document}